\documentclass[11pt]{article}
\usepackage{amsmath,amssymb}

\newtheorem{propo}{{\bf Proposition}}[section]
\newtheorem{coro}[propo]{{\bf Corollary}}
\newtheorem{lemma}[propo]{{\bf Lemma}} \newtheorem{theor}[propo]{{\bf
Theorem}} \newtheorem{ex}{{\sc Example}}[section]

\newenvironment{proof}{{\bf Proof.}}{$\Box$}

\def\N{{\mathbb N}}

\begin{document}

\vspace*{1.0in}

\begin{center} ON THE NILRADICAL OF A LEIBNIZ ALGEBRA
\end{center}
\bigskip

\begin{center} DAVID A. TOWERS 
\end{center}
\bigskip

\begin{center} Department of Mathematics and Statistics

Lancaster University

Lancaster LA1 4YF

England

d.towers@lancaster.ac.uk 
\end{center}
\bigskip

\begin{abstract} The purpose of this short note is to correct an error which appears in the literature concerning Leibniz algebras $L$: namely, that $N(L/I)=N(L)/I$ where $N(L)$ is the nilradical of $L$ and $I$ is the Leibniz kernel.
\end{abstract}

\section{Introduction}
\medskip

An algebra $L$ over a field $F$ is called a {\em Leibniz algebra} if, for every $x,y,z \in L$, we have
\[  [x,[y,z]]=[[x,y],z]-[[x,z],y]
\]
In other word,s the right multiplication operator $R_x : L \rightarrow L : y\mapsto [y,x]$ is a derivation of $L$. As a result, such algebras are sometimes called {\it right} Leibniz algebras, and there is a corresponding notion of {\it left} Leibniz algebra. Clearly, the opposite of a right Leibniz algebra is a left Leibniz algebra so, for our purposes, it does not matter which is used. Every Lie algebra is a Leibniz algebra and every Leibniz algebra satisfying $[x,x]=0$ for every element is a Lie algebra.
\par
 
Put $I=span\{x^2:x\in L\}$. Then $I$ is an ideal of $L$ and $L/I$ is a Lie algebra called the {\em liesation} of $L$. We define the following series:
\[ L^1=L,L^{k+1}=[L^k,L] \hbox{ and } L^{(1)}=L,L^{(k+1)}=[L^{(k)},L^{(k)}] \hbox{ for all } k=2,3, \ldots
\]
Then $L$ is {\em nilpotent} (resp. {\em solvable}) if $L^n=0$ (resp. $ L^{(n)}=0$) for some $n \in \N$. The {\em nilradical}, $N(L)$, (resp. {\em radical}, $R(L)$) is the largest nilpotent (resp. solvable) ideal of $L$.
\par

In \cite{gorb} it is claimed that $R(L/I)=R(L)/I$ and $N(L/I)=N(L)/I$. However, whereas the former is clearly true, the latter is false in general. The claim appears in the proof of \cite[Proposition 4]{gorb}, which has two corollaries. Although this paper appears only to have been published on arxiv it has been quite widely cited. Moreover, the results following from this assertion have been quoted in \cite[Proposition 4.4]{dms}, \cite[Propositions 3.3, 3.4 and Corollaries 3.5,3.6]{rak} and referenced to \cite{gorb}, and the same incorrect assertion is used to prove two results in a recent book (\cite[Proposition 2.1 and 2.2]{aor}). The results which Gorbatsevich uses this assertion to prove are true, as are \cite[Proposition 2.1 and 2.2]{aor}. The purpose of this short note is to give a correct characterisation of $N(L/I)$ and to provide or reference proofs for the results in which the incorrect assertion is used.
\par

Throughout, $L$ will denote a finite-dimensional (right) Leibniz algebra over a field $F$. The {\em Frattini ideal} of $L$ is the largest ideal of $L$ contained in every maximal subalgebra of $L$. We will denote algebra direct sums by $\oplus$ and vector space direct sums by $\dot{+}$. 

\section{The nilradical}
The literature concerning Leibniz algebras is quite diverse and a number of results have been duplicated, so the survey articles \cite{feld}, \cite{gorb} and the new book \cite{aor} are useful. The nilradical of a Leibniz algebra is a well-defined object.

\begin{theor} The sum of two nilpotent ideals of a Leibniz algebra is nilpotent.
\end{theor}
\begin{proof} See \cite[Theorem 5.14]{feld} or \cite[Lemma 1.5]{schunck}. 
\end{proof}

\begin{coro} Any Leibniz algebra has a maximal nilpotent ideal containing all nilpotent ideals of $L$.
\end{coro}
\begin{proof} A valid proof of this can also be found as  \cite[Corollary 4]{bosko}. Note that the proof given in \cite[Proposition 1]{gorb} and \cite[Proposition 2.1]{aor} is incorrect.
\end{proof}
\medskip

Next, we show that $N(L/I)\neq N(L)/I$ in general.

\begin{ex} Let $ L = Fx + Fx^2$ where $[x^2,x] = x^2$ is the two-dimensional solvable cyclic Leibniz algebra, we have $I = Fx^2 = N(L)$ and $L/I$ is the nilradical of $L/I$. 
\end{ex}

The fact that $N(L)=I$ is not significant in the above example as the following class of algebras shows.
\begin{ex} Let $L=Fx_1+\ldots +Fx_r+Fx_{r+1}+\ldots +Fx_n+Fy$ where $[x_i,y]=x_i$ for $r+1\leq i\leq n$ and all other products are zero. Then $I=Fx_{r+1}+\ldots +Fx_n$, $N(L)=Fx_1+\ldots +Fx_n$ and $N(L/I)=L/I$.
\end{ex}

Nor is the fact that $N(L/I)=L/I$ significant in the above examples as can be seen by taking the direct sum of them with a simple Lie algebra.

Next we look for more information on $N(L/I)$. First note the following.

\begin{lemma}\label{1} If $I\subseteq \phi(L)$ then $N(L/I)=N(L)/I$.
\end{lemma}
\begin{proof} Clearly $N(L)/I\subseteq N(L/I)=K/I,$ say. But $K$ is nilpotent, by \cite[Theorem 5.5]{barnes}.
\end{proof}

\begin{lemma}\label{2} If $I \not \subseteq \phi(L)$ then there is a subalgebra $B$ of $L$ such that $L=I+B$ and $I\cap B\subseteq \phi(B)$.
\end{lemma}
\begin{proof} This is \cite[Lemma 7.1]{frat}.
\end{proof}
\medskip

Then, using the same notation as in Lemma \ref{2}, we have the following.

\begin{theor} The nilradical of $L/I$ is $N(L/I)=(I+N(B))/I$ and this is the same as $N(L)/I$ if and only if $R_n|_I$ is nilpotent for all $n\in N(B)$.
\end{theor}
\begin{proof} We have
\[ N\left(\frac{L}{I}\right)\cong N\left(\frac{B}{I\cap B}\right)=\frac{N(B)}{I\cap B}=\frac{N(B)}{I\cap N(B)}\cong \frac{I+N(B)}{I}.
\] But $(I+N(B)/I\subseteq N(L/I)$, so equality results.
\par

Now, $I+N(B)=N(L)$ if and only if $N(B)$ acts nilpotently on the right on $I$. 
\end{proof}

\section{Some results where $N(L/I)=N(L)/I$ were used}
First we have the following analogue of a well-known result for Lie algebras. The only references to this in the literature of which we are aware are  \cite[Proposition 4]{gorb} and \cite[Proposition 2.2]{aor}. However, the proof in each case is incorrect, though the result is true.

\begin{propo}\label{3} Let $L$ be a Leibniz algebra over a field of characteristic zero, $R$ be its radical and $N$ its nilradical. Then $[L,R]\subseteq N$.
\end{propo}
\begin{proof} Since $N(L/\phi(L))=N(L)/\phi(L)$ we can assume that $L$ is $\phi$-free. Then $L=$Asoc$(L)\dot{+} V$, where $V=S\oplus Z(V)$ and $S$ is a semisimple Lie algebra, by \cite[Corollary 2.9]{stit}. Also $R=$Asoc$(L)+Z(V)$ and $N=$Asoc$(L)$, by \cite[Theorem 2.4]{stit}, from which the result is clear.
\end{proof}
\medskip

The above result has the following corollary which appears in several places in the literature. It occurs as \cite[Corollaries 5 and 6]{gorb} and \cite[Corollaries 2.2 and 2.3]{aor} but the proofs are incorrect. It appears with correct proofs as \cite[Corollary 6.8]{feld}, \cite[Corollary 3]{pak}, \cite[Theorem 4]{ao}, \cite[Theorem 2]{ns} and \cite[Theorem 2.6]{barnes}.

\begin{coro}(\cite[Corollary 5]{gorb}) With same notation as in Proposition \ref{3}, $[R,R]\subseteq N$; in particular, $[R,R]$ is nilpotent; in fact, $L$ is solvable if and only if $[L,L]$ is nilpotent. 
\end{coro}


\begin{thebibliography}{1}

\bibitem{ns} {\sc S. Albeverio, Sh. A. Ayupov and B.A. Omirov}, `On nilpotent and simple Leibniz algebras', {\em Comm. Alg.} {\bf 33 (1)} (2005), 159-172.

\bibitem{ao} {\sc Sh.A. Ayupov ,B.A.  Omirov}, `On Leibniz algebras', {\em Algebra and
Operators Theory, Proceeding of the Colloquium in Tashkent},
(1997), Kluwer Academic Publishers, (1998), 1–13.

\bibitem{aor} {\sc S. Ayupov, B. Omirov and I. Rakhimov}, `Leibniz Algebras, Structure and Classification', CRC Press, Taylor and Francis Group (2020).

\bibitem{barnes} {\sc D.W. Barnes}, `Some theorems on Leibniz algebras', {\em Comm. Algebra} {\bf 39(7)} (2011), 2463--2472.

\bibitem{schunck} {\sc D.W. Barnes}, 'Schunck classes of soluble Leibniz algebras', {\em Comm. Alg.} {\bf 41} (2013), 4046-4065.

\bibitem{stit} {\sc C. Batten, L. Bosko-Dunbar, A. Hedges, J.T. Hird, K. Stagg and E. Stitzinger}, `A Frattini theory for Leibniz algebras', {\em Comm. Alg.} {\bf 41(4)} (2013), 1547--1557.

\bibitem{bosko}, {\sc L. Bosko, A.  Hedges, J.T. Hird, N. Schwartz, K.  Stagg}, `Jacobson's refinement of Engel's theorem for Leibniz algebras', {\em Involve} {\bf 4 (3)} (2011), 293–296. 

\bibitem{dms} {\sc I. Demir, K.C. Misra, E. Stitzinger}, `On some structures of Leibniz algebras', in {\em Recent Advancesin Representation Theory, Quantum Groups, Algebraic Geometry, and Related Topics, ContemporaryMathematics}, {\bf 623} Amer. Math. Soc., Providence, RI, (2014), 41-54.

\bibitem{feld} {\sc J. Feldvoss}, `Leibniz algebras as nonassociative algebras', in:Nonassociative Mathematicsand Its Applications, Denver, CO, 2017 (eds. P. Vojtˇechovsky, M. R. Bremner, J. S. Carter, A. B. Evans, J. Huerta, M. K. Kinyon, G. E. Moorhouse, J. D. H. Smith), Contemp. Math.,vol.721, Amer. Math. Soc., Providence, RI, 2019, pp. 115–149. (This paper can also be obtained from the arXiv via arXiv:1802.07219.)

\bibitem{gorb} {\sc V.V. Gorbatsevich}, `On some basic properties of Leibniz algebras',arxiv:1302.3345v2(2013).

\bibitem{pak} {\sc A. Patsourakos}, `On nilpotent properties of Leibniz algebras', {\em Comm. Alg.} {\bf 35 (12)} (2007), 3828-3834.

\bibitem{rak} {\sc I. S. Rakhimov}, `On classification problem of Loday algebras', {\em Contemporary Mathematics} {\bf 672} (2016), 225-244.

\bibitem{frat} {\sc D.A. Towers}, `A Frattini theory for algebras', {\em
Proc. London Math. Soc.} (3) {\bf 27} (1973), 440--462.
\end{thebibliography}
\end{document}